\documentclass{amsart}
\usepackage{graphicx} 
\usepackage[dvipsnames]{xcolor}
\usepackage{mathrsfs}

\author{Neal Bez}
\address[Neal Bez]{Graduate School of Mathematical Sciences, The University of Tokyo,
3-8-1 Komaba, Meguro-ku, Tokyo 153-8914, Japan}
\email{bez@ms.u-tokyo.ac.jp}

\author{Anthony Gauvan}
\address[Anthony Gauvan]{Department of Mathematics, Graduate School of Science and Engineering,
Saitama University, Saitama 338-8570, Japan}
\email{a.gauvan@gmail.com}

\author{Hiroshi Tsuji}
\address[Hiroshi Tsuji]{Department of Mathematics, Institute of Science Tokyo, 2-12-1 Ookayama, Meguro-ku, Tokyo 152-8551, Japan}
\email{tsujihiroshi@math.sci.isct.ac.jp}

\usepackage{amsfonts, amssymb, amsmath, verbatim, amsthm, mathrsfs, amscd, enumerate, ascmac, fancyhdr, caption, hyperref, framed, subfigure}
\usepackage{bbm}
\numberwithin{equation}{section}

\usepackage{tikz}
\usetikzlibrary{positioning,intersections, calc, arrows,decorations.markings, angles}

\newtheorem{theorem}{Theorem}[section]
\newtheorem{lemma}[theorem]{Lemma}
\newtheorem*{lemma*}{Lemma}
\newtheorem{proposition}[theorem]{Proposition}

\theoremstyle{definition}

\newtheorem*{claim*}{Claim}

\theoremstyle{remark}
\newtheorem*{remark}{Remark}
\newtheorem*{overview}{Overview}

\thanks{
This work was supported by JSPS Kakenhi grant numbers 22H00098, 23K25777, 24H00024  (Bez), 
23KF0188 (Gauvan), and 24KJ0030 (Tsuji).
}

\begin{document}

\title[]{Operator capacity, the Brascamp--Lieb inequality and geometric programming}

\begin{abstract}
The capacity of completely positive operators and the Brascamp--Lieb constant can both be interpreted in terms of unconstrained geometric programming up to an additional minimisation over a compact group. We shine light on this perspective and make use of it to make novel contributions in both directions. For example, by making use of recent work of Bennett--Bez--Buschenhenke--Cowling--Flock, we prove new results regarding near-minimisers and local H\"older regularity of operator capacity. In addition, we observe that these results may be extended to the more general notion of capacity of quiver data. Furthermore, the geometric programming viewpoint allows us to give a new proof of the finiteness characterisation of the Brascamp--Lieb constant due to Bennett--Carbery--Christ--Tao (assuming Lieb's theorem on gaussian saturation).
\end{abstract}

\maketitle

\section{Background and motivation}
Given two integers $n$ and $m$, we shall say that a linear operator 
\begin{equation} \label{e:cp}
{\mathcal{T}}: \mathbb{C}^{n\times n} \to \mathbb{C}^{m\times m}
\end{equation}
is \emph{completely positive} if there is a finite family of complex matrices $ \mathcal{A} \subset \mathbb{C}^{m \times n} $ (Kraus operators) such that 
\[
{\mathcal{T}}(X) = \sum_{A \in \mathcal{A}}A X A^*
\]  
for all $X \in \mathbb{C}^{n \times n}$. Here, $A^* \in \mathbb{C}^{n \times m}$ denotes the conjugate-transpose of $A$. \emph{Positivity} of $\mathcal{T}$ is the property that $\mathcal{T}(X) \succeq 0$ whenever $X \succeq 0$, where $X \succeq 0$ means $X$ is positive semidefinite. Complete positivity is a stronger property which is perhaps more conventionally stated in terms of the positivity of tensor products of $\mathcal{T}$ with identity transformations\footnote{When $m = n$, a theorem of Choi \cite{Choi} shows this is equivalent to the existence of representation as in \eqref{e:cp} and, as observed in \cite{GGOW_GAFA}, this  extends to the case where $n$ and $m$ are not the same. We also remark that the representation using Kraus operators is not unique in general.}. 

A fundamental quantity associated with a completely positive operator \( {\mathcal{T}} \) is its \emph{capacity} given by\footnote{Up to a constant depending on $n$ and $m$, this is the definition as given in \cite{GGOW_GAFA}. In \cite{KLR}, up to constants, operator capacity is defined to be the $m$th root of the definition in this paper.}
\[
\mathrm{cap}({\mathcal{T}}) = \inf \left\{ \frac{\det({\mathcal{T}}(X))}{\det(X)^{m/n}} \,\middle|\, X \succ 0
\right\},
\] 
where $X \succ 0$ means that $X \in \mathbb{C}^{n \times n}$ is positive definite. The notion of capacity of a completely positive operator can be found in fundamental work of Gurvits \cite{Gurvits} in connection with determining the invertibility of so-called symbolic matrices (Edmonds' Problem). For this, Gurvits devised the operator scaling algorithm which alternately generates completely positive operators which are either row-stochastic or column-stochastic, a quantum analogue of the classical algorithm (often called the Sinkhorn algorithm) for generating doubly-stochastic matrices. The role of capacity in this context is to give a means of quantifying progress as the algorithm proceeds.

Building on the approach in \cite{Gurvits}, by using the operator scaling algorithm, Garg--Gurvits--Oliveira--Wigderson \cite{GGOW_FOCM} succeeded in showing that invertibility of symbolic matrices in non-commuting variables over $\mathbb{Q}$ can be tested in polynomial time. They actually showed that the operator scaling algorithm can be used to approximate capacity to any prescribed precision in polynomial time. In a subsequent paper by the same authors \cite{GGOW_GAFA}, it was observed that there is a remarkable connection between capacity of completely positive operators and the Brascamp--Lieb constant. Here, the terminology \emph{Brascamp--Lieb constant} refers to the optimal constant $C$ in the inequality
\begin{equation} \label{e:BL}
\int_{\mathbb{R}^n} \prod_{i=1}^\ell f_i(L_i x)^{\theta_i} \, \mathrm{d}x \leq C \prod_{i=1}^\ell \left( \int_{\mathbb{R}^{n_i}} f_i \right)^{\theta_i}.
\end{equation}
Here, the $\ell$-tuple of linear mappings \( \boldsymbol{L} = (L_i)_{i \in [\ell]} \) and $\ell$-tuple of non-negative exponents  \( \boldsymbol{\theta} = (\theta_i)_{i \in [\ell]} \) are given, and $(\boldsymbol{L},\boldsymbol{\theta})$ is called a \emph{Brascamp--Lieb datum}. The constant $C$ is independent of the non-negative integrable functions \( (f_i)_{i \in [\ell]} \). A fundamental theorem of Lieb \cite{Lieb} says that the optimal value of $C$, which we denote in the standard way by $\mathrm{BL}(\boldsymbol{L},\boldsymbol{\theta})$, is exhausted by centred gaussians and this leads to the expression
\begin{equation} \label{e:Lieb}
\mathrm{BL}(\boldsymbol{L},\boldsymbol{\theta})^{-2} = \inf \left\{ \frac{\det(\sum_{i=1}^\ell \theta_iL_i^*A_iL_i)}{\prod_{i=1}^\ell (\det (A_i))^{\theta_i} } \ \middle| \ A_i \succ 0, i \in [\ell] \right\}.
\end{equation}
For clarification, in the above expression $A_i \succ 0$ means that $A_i \in \mathbb{R}^{n_i \times n_i}$ is positive definite. It will always be clear from the context whether $\succ$ means complex positive definite or real positive definite.

Under the assumption that each exponent $\theta_i$ is rational, it was shown in \cite[Lemma 4.4]{GGOW_GAFA} that 
\begin{equation} \label{e:BLascap}
\mathrm{BL}(\boldsymbol{L},\boldsymbol{\theta})^{-2} = \mathrm{cap}(\mathcal{T}_{(\boldsymbol{L},\boldsymbol{\theta})})
\end{equation}
for a certain completely positive operator $\mathcal{T}_{(\boldsymbol{L},\boldsymbol{\theta})}$ explicitly constructed from the datum $(\boldsymbol{L},\boldsymbol{\theta})$. Building on theory developed for capacity in \cite{GGOW_FOCM}, this bridge gave rise to new insight and powerful results on algorithmic aspects of the Brascamp--Lieb inequality in \cite{GGOW_GAFA}. This connection seems likely to continue to enrich the theory of the Brascamp--Lieb inequality, particularly from an algorithmic or computational viewpoint; see, for example, work by Franks \cite{Franks} and Kwok--Lau--Ramachandran \cite{KLR} for  developments in this direction which have already occurred post \cite{GGOW_GAFA}. We would also like to point out that some arguments in \cite{GGOW_GAFA} do not rely on the datum being rational and indeed key ideas in \cite{GGOW_GAFA} heavily inspired our recent work on  ubiquity of so-called geometric Brascamp--Lieb data \cite{BGT}. 

From an algorithmic or computational stance, assuming that $\boldsymbol{L}$ and/or $\boldsymbol{\theta}$ have rational components is natural. On the other hand, the Brascamp--Lieb inequality is extremely far-reaching and has enriched aspects of a variety of fields including harmonic analysis, convex geometry, information theory, frame theory, number theory and scattering theory \cite{AFR, Ball, BGMN, Brazitikos, B, CC-E, CL, GZ, HM, P, Petrow, Zhang}. Broadly speaking, a restriction to rational Brascamp--Lieb data will often present limitations in such contexts and in many applications of Brascamp--Lieb inequalities one is happy to accept information of a more qualitative nature. 

With the above in mind, one of our goals in this paper is to draw attention to the fact that \emph{both} operator capacity and the inverse-square of the Brascamp--Lieb constant can be naturally viewed as special cases of \emph{unconstrained geometric programming}\footnote{See Section \ref{section:GP} for the meaning of this terminology from optimisation theory.} up to a minimisation over a compact group. This is by no means radically new and one can find papers on either capacity-related topics, or Brascamp--Lieb theory, which have, at some level, made use of the viewpoint. As will become clear, we have in mind, in particular, the papers by B\"urgisser--Li--Nieuwboer--Walter \cite{BLNW}, Gurvits \cite{Gurvits} and Straszak--Vishnoi \cite{SV} on the capacity side, and Barthe \cite{Barthe} and Bennett--Bez--Buschenhenke--Cowling--Flock \cite{BBBCF} on the Brascamp--Lieb inequality side. As we shall explain in more detail later, one can find even find some overlapping results in these papers of a similar nature which have been independently proved; this is, of course, unsurprising given the relatively recent emergence of the connection between capacity and the Brascamp--Lieb inequality and that the underlying motivation for the papers is significantly different. 

By drawing attention to the connections between capacity, the Brascamp--Lieb constant and geometric programming, our hope is that the present paper will be useful to all communities to which these fundamental quantities are of interest. We shall apply this viewpoint and obtain new results concerning near-minimisers and local H\"older continuity of capacity, thereby extending in certain respects some recent work of Garg--Gurvits--Oliveira--Wigderson \cite{GGOW_FOCM} and Allen-Zhu--Garg--Li--Oliveira--Wigderson \cite{AGLOW}. We also observe that our approach can be used to establish local H\"older continuity of capacity associated with quiver data\footnote{This notion of capacity, introduced by Chindris--Derksen \cite{CD1, CD}, encompasses both operator capacity and the inverse-square of the Brascamp--Lieb constant -- see the end of Section \ref{section:CapGP} for further details.}. As an example of an application, establishing regularity results on such capacities should prove to be very useful in the sense that, when trying to prove certain properties of capacity for arbitrary feasible data, regularity (in particular, continuity) facilitates a reduction to extremisable data by density-type considerations; examples of this can already be seen in, for example, \cite{BGT} and \cite{CD}.

Furthermore, we shall use the geometric programming viewpoint to obtain a new proof of a well-known theorem by Bennett--Carbery--Christ--Tao \cite{BCCT, BCCT_MRL} regarding the finiteness of the Brascamp--Lieb constant. This naturally leads to a question regarding the shape of the Brascamp--Lieb polyhedron\footnote{For a given $\boldsymbol{L}$, this is the space of all $\boldsymbol{\theta}$ such that the associated Brascamp--Lieb constant is finite.} which we settle in the case of Brascamp--Lieb data studied by Finner \cite{Finner}.

\begin{overview} In Section \ref{section:GP} we introduce unconstrained geometric programming and present some general theory which is pertinent for the rest of the paper. 
Section \ref{section:CapGP} is concerned with operator capacity and capacity of quiver data. Finally, Section \ref{section:BLGP} is concerned with the Brascamp--Lieb constant.
\end{overview}

\section{Unconstrained geometric programming} \label{section:GP}

Given a finite index set \( J \), a collection of vectors $ \boldsymbol{u} = (u_j)_{j \in J} \subset \mathbb{R}^n$ and a collection of non-negative real numbers
$ \boldsymbol{d} = (d_j)_{j \in J}$, the optimisation problem
\begin{align*}
\text{minimise} \quad & \sum_{j \in J} d_j e^{\langle y, u_j  \rangle} \\
\text{subject to} \quad & y \in \mathbb{R}^n
\end{align*}
is an example of an optimisation problem known as \textit{unconstrained geometric programming}. More general geometric programs include further constraints on the variable and are widely studied in the optimisation theory literature (see, for example, \cite{BV, C, DPZ}), but in this paper we shall only be concerned with the unconstrained version. 

To see the connection to capacity, if we first restrict attention to a diagonal input $D = \mathrm{diag}(\lambda_1,\ldots,\lambda_n)$, then $\det(\mathcal{T}(D))$ is a polynomial in $\lambda$. Writing $\mathrm{cap}_0$ for the restricted capacity to diagonal inputs, upon reparametrising $\lambda_j = e^{y_j}$, this means
\begin{equation*} \label{e:key}
\mathrm{cap}_0(\mathcal{T}) = \inf_{y \in \mathbb{R}^n} \sum_{j \in J} d_j(\mathcal{T}) e^{\langle y, u_j \rangle}
\end{equation*}
for a certain collection of vectors $(u_j)_{j \in J} \subset \mathbb{R}^n$ which are fixed (depending only on $m$ and $n$), and the  coefficients \( \boldsymbol{d} = (d_j(\mathcal{T}))_{j \in J} \) are determined by the operator \( \mathcal{T} \). As we shall see, the fact that $\mathcal{T}$ is completely positive means that each $d_j(\mathcal{T})$ is non-negative. We refer the reader forward to Section \ref{section:CapGP} for the details.


For the Brascamp--Lieb constant, note that if each $L_i$ is a rank-one linear mapping and we write $L_ix = \langle x,v_i \rangle$, then one may use the Cauchy--Binet formula to write
\begin{equation} \label{e:BLrank1}
\mathrm{BL}(\boldsymbol{L},\boldsymbol{\theta})^{-2} = \inf_{y \in \mathbb{R}^m} \sum_{I \in \mathcal{I}} \theta_I \det(v_i)_{i \in I}^2 e^{\langle 1_I - \boldsymbol{\theta},y \rangle}.
\end{equation}
Here, $\mathcal{I}$ is the set of all subsets of $[m]$ of cardinality $n$, and $\theta_I := \prod_{i \in I} \theta_i$. The expression \eqref{e:BLrank1} is due to Barthe \cite{Barthe}. In the case of general-rank mappings, by diagonalising the $A_i \succ 0$ appearing in \eqref{e:Lieb}, it is possible to obtain an expression in a similar spirit to \eqref{e:BLrank1} but with an additional minimisation over a product of orthogonal groups. This idea was first used in \cite{BBCF} in order to establish continuity of the Brascamp--Lieb constant with respect the linear mappings. Details can be found in Section \ref{section:BLGP}.

For the objective function, we use the notation
\[
\Phi_{\boldsymbol{d}}(y) := \sum_{j \in J} d_j e^{\langle y, u_j \rangle}
\]
and for its infimum we write
\[
\Psi(\boldsymbol{d}) := \inf_{y \in \mathbb{R}^n} \Phi_{\boldsymbol{d}}(y).
\]
To keep the notation as light as possible, here we  suppress the dependence on the vectors $\boldsymbol{u}$. Also we remark that if we set 
\[
J_+ := \{ j \in J : d_j > 0\},
\]
then clearly it is possible to restrict attention to $j \in J_+$. Nevertheless, as we shall see, it is important to permit coefficients which may vanish. In particular, it will be important to have uniform diameter bounds on near-minimisers which do not depend on the size of the smallest nonzero coefficient. 

The set $\{u_j : j \in J_+\}$ can be regarded as a Newton polytope. Indeed, if we set $\lambda_j := e^{y_j}$ then
\[
\sum_{j \in J} d_j e^{\langle y, u_j \rangle} = \sum_{j \in J_+} d_j \lambda^{u_j},
\]
where  
$
\lambda^a := \prod_{\ell=1}^n \lambda_\ell^{a_\ell}
$
for vectors $\lambda = (\lambda_1,\ldots,\lambda_n) \in [0,\infty)^n$ and $a = (a_1,\ldots,a_n) \in \mathbb{R}^n$. Thus, the reparametrised objective function is a polynomial if each $u_{j} \in \mathbb{N}_0^n$ and in this case the set $\{u_j : j \in J_+\}$ is the Newton polytope of this polynomial. For general $u_{j} \in \mathbb{R}^n$, such a quantity is known as a  \emph{posynomial} and one may extend the standard notion of the Newton polytope of a polynomial to such functions.

\subsection{Feasibility}
\begin{proposition} \label{p:GPfeasible}
Suppose $J_+ \neq \emptyset$. Then $\Psi({\boldsymbol{d}}) > 0$ holds if and only if $0 \in  \mathrm{conv} \{u_j : j \in J_+\}$. 
\end{proposition}
Although this result seems to be well-known in the optimisation community, we provide the details.
\begin{proof}
Suppose first that $0 = \sum_{j \in J} t_ju_j$ where $\sum_{j \in J} t_j = 1$, $t_j \geq 0$ and $t_j = 0$ if $j \notin J_+$. For any $y \in \mathbb{R}^n$,
\begin{align*}
    \Phi_{\boldsymbol{d}}(y) \geq \sum_{\substack{j \in J \\ t_j > 0}} d_je^{\langle y,u_j \rangle} \geq \prod_{\substack{j \in J \\ t_j > 0}} \bigg(\frac{d_j}{t_j} e^{\langle y,u_j \rangle}\bigg)^{t_j} = \prod_{\substack{j \in J \\ t_j > 0}} \bigg(\frac{d_j}{t_j} \bigg)^{t_j}.
\end{align*}
It follows that
\begin{align*}
    \Psi(\boldsymbol{d}) \geq \prod_{\substack{j \in J \\ t_j > 0}} \bigg(\frac{d_j}{t_j} \bigg)^{t_j} > 0
\end{align*}
since $j \in J_+$ whenever $t_j > 0$.

Conversely, suppose $0 \notin  \mathrm{conv} \{u_j : j \in J_+\}$. By the hyperplane separation theorem, there exists $\omega \in \mathbb{R}^n$ such that $\langle \omega,u_j \rangle < 0$ for all $j \in J_+$ (this may also be viewed as an application of Farkas' lemma from optimisation theory). But then we clearly have
$\Phi_{\boldsymbol{d}}(s\omega) \to 0$ as $s \to \infty$, and hence $\Psi(\boldsymbol{d}) = 0$.
\end{proof}

\subsection{Near minimisers}
The objective function $y \mapsto \Phi_{\boldsymbol{d}}(y)$ is convex (in fact, log-convex). However, the domain is unbounded and therefore one cannot necessarily expect its infimum to be attained\footnote{In fact, it is known that the infimum is attained if and only if the relative interior of the Newton polytope contains the origin.}. Thus, the problem of unconstrained geometric programming is often viewed as the problem of obtaining arbitrarily good approximations to the infimum. From a computational point of view, this means obtaining algorithms which, for any small $\delta> 0$ return a point $y_\delta \in \mathbb{R}^n$ such that
\[
\Phi_{\boldsymbol{d}}(y_\delta) \leq \Psi({\boldsymbol{d}}) + \delta.
\]
We refer the reader to recent work of B\"urgisser--Li--Nieuwboer--Walter \cite{BLNW} for more in depth discussion on the computational aspects of unconstrained geometric programming, as well as areas of mathematics to which such results enjoy applications. 

Tightly connected is the problem of obtaining size thresholds (diameter bounds) which guarantee the existence of near-minimisers. In the context of the Brascamp--Lieb inequality, establishing such bounds was key to the work \cite{BBBCF} which gave a solution to a longstanding open problem regarding an extension of \eqref{e:BL} in which the linear mappings $L_i$ are replaced by certain nonlinear perturbations (the so-called nonlinear Brascamp--Lieb inequality). The following is \cite[Theorem 1.5]{BBBCF}.
\begin{theorem} \label{t:GPnearmin} \cite{BBBCF}
For any finite index set $J$ and collection of vectors $(u_j)_{j \in J} \subset \mathbb{R}^n$ there exist $N \in \mathbb{N}$ and $\delta_0 > 0$ such that the following holds. For any $\boldsymbol{d} = (d_j)_{j \in J} \subseteq [0,\infty)$ and any $\delta \in (0,\delta_0)$ there exists $y \in \mathbb{R}^n$ such that 
\[
|y| \leq N \log\bigg(\frac{1}{\delta}\bigg)
\]
and
\[
\Phi_{\boldsymbol{d}}(y) \leq \Psi({\boldsymbol{d}}) + \delta \max_{j \in J} d_j.
\]
\end{theorem}
This theorem was used in \cite{BBBCF} to show a that $\delta$-near-minimiser for the quantity on the right-hand side of \eqref{e:Lieb} exists with $\max\{\|A_i\|,\|A_i^{-1}\|\} \leq \delta^{-N}$; in other words, an effective form of Lieb's theorem for the Brascamp--Lieb constant. This version of Lieb's theorem played a significant role in an induction-on-scales argument in \cite{BBBCF} leading to a nonlinear version of the Brascamp--Lieb inequality. Theorem \ref{t:GPnearmin} was also used to establish the local H\"older regularity of $\boldsymbol{L} \mapsto \mathrm{BL}(\boldsymbol{L},\boldsymbol{\theta})$ via the forthcoming Theorem \ref{t:GPHolder}.

We emphasise that the above-mentioned applications of Theorem \ref{t:GPnearmin} rely on the uniformity of the statement in relation to the non-negative coefficients $(d_j)_{j \in J}$. Recently, statements somewhat similar to Theorem \ref{t:GPnearmin} have appeared in work by Straszak--Vishnoi \cite[Theorem 7]{SV} and B\"urgisser--Li--Nieuwboer--Walter \cite[Theorem 2.4]{BLNW}\footnote{See the forthcoming remark which provides some further context.}, however these results are not uniform in this way and the diameter bound on the near-minimiser goes to infinity as $\min_{j \in J} d_j$ approaches zero. The proof of Theorem \ref{t:GPnearmin} actually rests on a preliminary version for which the diameter bound blows up as the smallest coefficient approaches to zero (see \cite[Lemma 4.4]{BBBCF}) which is then upgraded to the uniform statement in Theorem \ref{t:GPnearmin} via an argument which splits the coefficients into two collections, ``small" and ``big", with a certain quantitative gap between them which is generated using a pigeonholing argument.

Regarding $N$ and $\delta_0$ in Theorem \ref{t:GPnearmin}, as pointed out after the proof of Theorem 1.5 in \cite{BBBCF}, an explicit choice can be extracted from the argument. In particular, the quantity $N$ depends in an interesting geometric manner on the vectors $(u_j)_{j \in J}$. This is the case too for the corresponding quantities in \cite[Theorem 7]{SV} and \cite[Theorem 2.4]{BLNW}, with the latter result being more general, but it is not clear to us how to compare these to $N$ arising from the argument in \cite{BBBCF}. 

\begin{remark}
Maximum entropy distributions in a discrete setting are of importance in machine learning and related areas. These distributions are solutions of maximum entropy programs of the form
\begin{align*}
\text{maximise} \quad & \sum_{j \in J} p_j \log \frac{d_j}{p_j} \\
\text{subject to} \quad & \sum_{j \in J} p_j u_j = \theta, \quad \sum_{j \in J} p_j = 1, \quad p_j \geq 0 \,\, (j \in J),
\end{align*}
where $(u_j)_{j \in J}$ and $\theta$ are given vectors in $\mathbb{R}^n$. These optimisation problems are in a sense dual to unconstrained geometric programs. To see this, we write \( \mathcal{P}(J) \) for the set of probability vectors supported on \( J \), and write
$\mathcal{P}(J,\theta)$ for those $\boldsymbol{p} \in \mathcal{P}(J)$ which satisfy the moment condition $\sum_{j \in J} p_j u_j = \theta$. One way to see the connection to entropy is to observe that the Legendre transform of $\Phi_{\boldsymbol{d}}$ is given by
\[
\left( \log \Phi_{\boldsymbol{d}} \right)^*(y) = \min_{\boldsymbol{p} \in \mathcal{P}(J, y)} \mathrm{KL}(\boldsymbol{p} \| \boldsymbol{d})
\]
for every \( y \in \mathrm{conv}(\boldsymbol{u}) \), where
\[
\mathrm{KL}(\boldsymbol{p} \| \boldsymbol{d}) := \sum_{j \in J} p_j \log \frac{p_j}{d_j}
\]
is the Kullback--Leibler divergence\footnote{Alternatively, one may also check that $\log \Phi_{\boldsymbol{d}}(y) = \sup_{\boldsymbol{p} \in \mathcal{P}(J)} (\langle y,\sum_{j \in J} p_ju_j \rangle - \mathrm{KL}(\boldsymbol{p} \| \boldsymbol{d}))$.}. In particular, we have
\begin{align*}
\log \Psi(\boldsymbol{d}) = - \left( \log \Phi_{\boldsymbol{d}} \right)^*(0) = \max_{\boldsymbol{p} \in \mathcal{P}(J, 0)} \sum_{j \in J} p_j \log \frac{d_j}{p_j}.
\end{align*}
As a result, for any marginal vector $\theta \in \mathbb{R}^n$, the above maximum entropy program is dual to
\[
\inf_{y \in \mathbb{R}^n} \log \sum_{j \in J} d_j e^{\langle y,u_j - \theta \rangle}.
\]
Capitalising on this connection, in the case of lattice vectors $(u_j)_{j \in J} \subseteq \mathbb{Z}^n$, Straszak--Vishnoi \cite{SV} were able to establish certain stability results for maximum entropy distributions by building on their analysis of diameter bounds for near-minimisers of unconstrained geometric programs (for lattice vectors). Such diameter bounds were later extended by B\"urgisser--Li--Nieuwboer--Walter \cite{BLNW} to allow for arbitrary $(u_j)_{j \in J}$ in $\mathbb{R}^n$.
\end{remark}

\subsection{Regularity}

Establishing the continuity of $\Psi$ on the set of $\boldsymbol{d} = (d_j)_{j \in J}$ can be shown as follows. The infimum (as the parameter varies) of a family of upper semicontinuous functions is upper semicontinuous, and thus it suffices to show that $\Psi$ is lower semicontinuous. The fact may be seen directly as follows. First, fix $\boldsymbol{d}$ and as above let $J_+ := \{j \in J : d_j > 0\}$. Take any $\delta \in (0, \delta_0)$, where $\delta_0 := \min_{j \in J_+} d_j$.  If $\|\boldsymbol{d} - \boldsymbol{d}'\|_\infty < \delta$, then
\begin{align*}
\Phi_{\boldsymbol{d}'}(y) \geq \sum_{j \in J_+} d_j'e^{\langle y,u_j \rangle} & \geq  \sum_{j \in J_+} (d_j - \delta) e^{\langle y,u_j \rangle}  \\
& \geq   \left(1 - \frac{\delta}{\delta_0}\right) \sum_{j \in J_+} d_j  e^{\langle y,u_j \rangle} = \left(1 - \frac{\delta}{\delta_0}\right) \Phi_{\boldsymbol{d}}(y)
\end{align*}
holds for all $y \in \mathbb{R}^n$. Therefore
\[
\Psi(\boldsymbol{d}') \geq \left(1 - \frac{\delta}{\delta_0}\right)\Psi(\boldsymbol{d})
\]
and it follows that $\Psi$ is lower semicontinuous, as claimed.

In fact, $\Psi$ is locally H\"older continuous. This fact was established in \cite[Theorem 4.5]{BBBCF} and the argument rested in a crucial way on Theorem \ref{t:GPnearmin}.
\begin{theorem} \cite{BBBCF}
\label{t:GPHolder}
For any finite index set $J$ and collection of vectors $(u_j)_{j \in J} \subset \mathbb{R}^n$ there exists $\alpha \in (0,1)$ such that the following holds. For any $R > 0$, there exists $C < \infty$ such that 
$$\left| \Psi(\boldsymbol{d}) - \Psi(\boldsymbol{d'}) \right| 
\leq C \left\| \boldsymbol{d} - \boldsymbol{d}' \right\|_\infty^\alpha
$$
whenever $\|\boldsymbol{d}\|_\infty, \|\boldsymbol{d}'\|_\infty \leq R$.
\end{theorem}




\section{Capacity} \label{section:CapGP}

\subsection{Capacity in terms of geometric programming} 

We denote by \[\mathscr{L}_{n,m} := \left\{ \mathcal{L}: \mathbb{C}^{n \times n} \to \mathbb{C}^{m \times m} \ \middle|\ \mathcal{L} \text{ is linear} \right\}\]
the Euclidean space of linear operators from \( \mathbb{C}^{n \times n} \) to \( \mathbb{C}^{m \times m} \). We equip each space of matrices $\mathbb{C}^{n \times n}$ and $\mathbb{C}^{m \times m}$ with the Hilbert--Schmidt inner product
\[
\langle X,Y \rangle := \mathrm{Tr}(X^*Y)
\]
and equip $\mathscr{L}_{n,m}$ with the induced operator norm, that is,
\[
\| \mathcal{L} \| := \sup \left\{ \frac{\| \mathcal{L}(X) \|}{\|X\|} \ \middle| \ X \in \mathbb{C}^{n \times n} \setminus \{0\} \right\}.
\] 
We write $\mathscr{T}_{n,m}$ for the set of all completely positive operators $\mathcal{T}: \mathbb{C}^{n \times n} \to \mathbb{C}^{m \times m}$. Since \( \mathscr{T}_{n,m} \subset \mathscr{L}_{n,m} \), it inherits the induced topology. 

Given a completely positive operator \( \mathcal{T} \in \mathscr{T}_{n,m} \), we define its \textit{diagonal capacity} by
\[
\mathrm{cap}_0({\mathcal{T}}) = \inf \left\{ \frac{\det({\mathcal{T}}(\Lambda))}{\det(\Lambda)^{m/n}} \,\middle|\, \Lambda \in \mathscr{D}_n, \Lambda \succ 0
\right\}.
\]
Here we are using the notation $\mathscr{D}_{n}$ for the set of diagonal matrices in $\mathbb{C}^{n \times n}$. In addition, for any unitary matrix \( U \in \mathrm{U}(n) \), we define the unitary conjugate \( \mathcal{T}_U \in \mathscr{T}_{n,m} \) of $\mathcal{T}$ by
\[
\mathcal{T}_U(X) := \mathcal{T}(UXU^*).
\]
It is easy to check that capacity and diagonal capacity are related by
\begin{equation} \label{e:rotation}
\mathrm{cap}(\mathcal{T}) = \inf_{U \in \mathrm{U}(n)} \mathrm{cap}_0(\mathcal{T}_U).
\end{equation}

Next we show that $\mathrm{cap}_0$ can be directly expressed as a geometric program. For this, first we take any $\mathcal{T} \in \mathscr{T}_{n,m}$, $\Lambda = \mathrm{diag}(\lambda_1, \dots, \lambda_n) \in \mathscr{D}_n$, and use mixed discriminants to write $\det (\mathcal{T}(\Lambda))$ as a polynomial in $\lambda$ with non-negative coefficients. Indeed, clearly we have 
\begin{align*}
\mathcal{T}(\Lambda) = \sum_{\ell = 1}^n \lambda_\ell T_\ell,
\end{align*}
where
$
T_\ell := \mathcal{T}(E_{(\ell,\ell)})
$
and $\{E_y : y \in [n]^2\}$ denotes the standard orthonormal basis (with respect to the Hilbert--Schmidt inner product) of $\mathbb{C}^{n \times n}$ consisting of matrix units\footnote{That is to say, \( E_y \) denotes the matrix with a 1 in position \( y \in [n]^2 \) and 0 elsewhere}. It follows that
\begin{equation} \label{e:detTLambda}
\det(\mathcal{T}(\Lambda)) = \sum_{j \in J_{n,m}} d_j(\mathcal{T})\lambda^j.
\end{equation}
Here, the coefficients $d_j(\mathcal{T})$ are given by
\begin{equation} \label{e:djdefn}
 d_j(\mathcal{T}) := \frac{m!}{j_1!\cdots j_n!} D(\overbrace{T_1,\ldots,T_1}^{j_1},\ldots,\overbrace{T_n,\ldots,T_n}^{j_n}),
\end{equation}
the index set $J_{n,m}$ is given by
\[
J_{n,m} := \{j \in \mathbb{N}_0^n : j_1+\cdots+j_n = m\}
\]
and, following standard multi-index notation, 
\[
\lambda^j := \prod_{i=1}^m \lambda_i^{j_i}.
\]
Also, $D$ denotes the mixed discriminant, a quantity which provides a simultaneous generalisation of the matrix determinant and matrix permanent. It may be defined, for example, by
\begin{equation} \label{e:Ddefn}
D(B_1,\ldots,B_n) := \frac{1}{n!} \sum_{\sigma \in S_n} \det(b_{ij}^{\sigma(j)})
\end{equation}
for matrices $B_\ell = (b^\ell_{ij})_{i,j \in [n]}  \in \mathbb{C}^{n \times n}$, $\ell \in [n]$, and where $S_n$ denotes the set of all permutations of $[n]$.

In the case $m = n$, \eqref{e:detTLambda} follows immediately from \cite[Lemma 1]{Bapat}. The general case follows easily by combining \cite[Lemma 1]{Bapat} with multilinearity and permutation-invariance properties of mixed discriminants.

From \eqref{e:detTLambda} and the reparametrisation $\lambda_\ell = e^{y_\ell}$, $y_\ell \in \mathbb{R}$, we have
\begin{equation} \label{e:cap0GP}
\det(\mathcal{T}(\Lambda)) = \sum_{j \in J_{n,m}} d_j(\mathcal{T}) e^{\langle j,y\rangle}.
\end{equation}
Consequently
\begin{equation} \label{e:captoGPmain}
\frac{\det(\mathcal{T}(\Lambda))}{\det(\Lambda)^{m/n}} =  \sum_{j \in J_{n,m}} d_j(\mathcal{T}) e^{\langle u_j,y \rangle},
\end{equation}
where the family of vectors \( \boldsymbol{u} = (u_j)_{j \in J_{n,m}} \subset \mathbb{R}^n \) is given by
\[
u_j := j - \frac{m}{n} \boldsymbol{1}_n. 
\]
Crucially, the coefficients $d_j(\mathcal{T})$ are non-negative. Indeed, $T_\ell \succ 0$ for each $\ell \in [n]$ since $\mathcal{T} \in \mathscr{T}_{n,m}$, and hence $d_j(\mathcal{T}) \geq 0$ follows from \cite[Lemma 2]{Bapat}. In summary, we have shown the following.
\begin{proposition} \label{p:captoGP}
For any \( \mathcal{T} \in \mathscr{T}_{n,m} \), we have
\[
\mathrm{cap}_0(\mathcal{T}) = \Psi(\boldsymbol{d}(\mathcal{T})),
\]
where the non-negative coefficients $\boldsymbol{d}(\mathcal{T}) = (d_j(\mathcal{T}))_{j \in J_{n,m}}$ are given by \eqref{e:djdefn}. Consequently
\[
\mathrm{cap}(\mathcal{T}) = \inf_{U \in \mathrm{U}(n)} \Psi(\boldsymbol{d}(\mathcal{T}_U)).
\]
\end{proposition}
For later use, we also observe that
\begin{equation} \label{e:djbound}
d_j(\mathcal{T}) \leq C_{n,m}\|\mathcal{T}\|^m
\end{equation}
which follows from \eqref{e:djdefn}, \eqref{e:Ddefn} and Hadamard's inequality.

\begin{remark}
After preparing an early version of this article, we became aware that the content of this section can be found in work of Gurvits--Samorodnitsky \cite{GS} and Gurvits \cite{Gurvits}. For instance, the expression \eqref{e:cap0GP} with the coefficients interpreted as mixed discriminants can be found in \cite{GS}. Also, one can find use of \eqref{e:rotation} in \cite[Lemma 4.5]{Gurvits} in which Gurvits proves the fundamental fact that, in the case $m = n$, the strict positivity of $\mathrm{cap}(\mathcal{T})$ is equivalent to the property that $\mathcal{T}$ is rank non-decreasing (later this was extended to general $n$ and $m$ in \cite{GGOW_GAFA}).
\end{remark}

\subsection{Near-minimisers and regularity of capacity}


The fact that capacity is continuous\footnote{Here, we consider use the usual euclidean metric on the space of positive definitive matrices.} was originally observed without proof by Gurvits \cite{Gurvits}. Upper semicontinuity of capacity follows easily from abstract arguments since it is the infimum of a family of continuous functions. Lower semicontinuity follows from the continuity of $\Psi$. Indeed, since $\mathcal{T} \mapsto d_j(\mathcal{T})$ is clearly continuous, the continuity of $\mathrm{cap}_0$ follows directly from Proposition \ref{p:captoGP} and the continuity of $\Psi$. Now suppose $\| \mathcal{T}_\ell - \mathcal{T}\| \to 0$ as $\ell \to \infty$, and take a subsequence $(\mathcal{T}_{\ell_k})_k$ for which $\liminf_{\ell \to \infty} \mathrm{cap}(\mathcal{T}_\ell) = \lim_{k \to \infty} \mathrm{cap}(\mathcal{T}_{\ell_k})$. Set $\mathrm{c}(\mathcal{T},U) := \mathrm{cap}_0(\mathcal{T}_U)$ and note that $\mathrm{c}$ is jointly continuous. By compactness, we obtain a sequence $(U_{\ell_k})_k \in \mathrm{U}(n)$ for which
\[
\mathrm{cap}(\mathcal{T}_{\ell_k}) = \mathrm{c}(\mathcal{T}_{\ell_k},U_{\ell_k}).
\]
Again by compactness, we know that there is a subsequence $(U_{\ell_{k_j}})_j$ which converges to some $U \in \mathrm{U}(n)$. Then
\[
\lim_{j \to \infty} \mathrm{cap}(\mathcal{T}_{\ell_{k_j}}) = \lim_{j \to \infty} \mathrm{c}(\mathcal{T}_{\ell_{k_j}},U_{\ell_{k_j}}) = \mathrm{c}(\mathcal{T},U) \geq \mathrm{cap}(\mathcal{T})
\]
and it follows that $\liminf_{\ell \to \infty} \mathrm{cap}(\mathcal{T}_\ell) \geq \mathrm{cap}(\mathcal{T})$.

In terms of quantitative forms of continuity, as far as we are aware, the first result of this type is the following due to Garg--Gurvits--Oliveira--Wigderson \cite[Theorem 4.5]{GGOW_FOCM} for rational operators.
\begin{theorem} \label{t:GGOW} \cite{GGOW_FOCM}
Suppose \( \mathcal{A} = (A_1, \dots, A_N) \), \( \mathcal{A}'=(A_1', \dots, A_N') \)  with $A_i, A_i' \in \mathbb{Q}^{n \times n}$ for each $i \in [N]$, and let $\mathcal{T}_\mathcal{A}$, $\mathcal{T}_\mathcal{A'}$ be the associated completely positive operators. Let $b$ denote the bit complexity of the entries of the matrices in $\mathcal{A}$. Then there exists a polynomial \( P(n, b, \log N) \) such that if \( \delta \leq \exp(-P(n, b, \log N)) \) then
\[
|\mathrm{cap}(\mathcal{T}_\mathcal{A}) - \mathrm{cap}(\mathcal{T}_{\mathcal{A}'})| \leq \frac{P(n, b, \log N)}{\log(1/\delta)^{1/3}} \mathrm{cap}(\mathcal{T}_{\mathcal{A}'})
\]
whenever $\|A_i - A_i'\|\leq \delta$ for each $i \in [N]$.
\end{theorem}
In this section, we shall show that capacity  possesses significantly greater regularity than this, and our main result provides local H\"older continuity. As observed by Gurvits \cite{Gurvits}, differentiability fails in general. It turns out that our approach is rather robust insofar as it does not require us to impose a restriction on the operators to be rational and, additionally, our statement is free of explicit reference to the Kraus operators.

\subsubsection{A diameter bound for near-minimisers}

Write
\[
\mathrm{cap}(\mathcal{T};X) :=  \frac{\det({\mathcal{T}}(X))}{\det(X)^{m/n}}.
\]
\begin{theorem} \label{t:capnearmin}
There exists $N \in \mathbb{N}$, $\delta_0 > 0$ and $C_0 > 0$, depending on at most $n$ and $m$, such that the following holds. For all $\mathcal{T} \in \mathscr{T}_{n,m}$ and $\delta \in (0,\delta_0)$ there exists $X_\delta \succ 0$ such that
\[
\mathrm{cap}(\mathcal{T};X_\delta) \leq \mathrm{cap}(\mathcal{T}) + C_0\|\mathcal{T}\|^m\delta \qquad \text{and} \qquad \max\{\|X_\delta\|,\|X_\delta^{-1}\|\} \leq \delta^{-N}.
\]
\end{theorem}
\begin{proof}
 We use Theorem \ref{t:GPnearmin} with the family of vectors $\boldsymbol{u} = (u_j)_{j \in J_{n,m}}$ given by $u_j := j - \frac{m}{n}\boldsymbol{1}_n$, and we let $N \in \mathbb{N}$ and $\delta_0' > 0$ be those which are given to us by that theorem. 

Now fix $\mathcal{T} \in \mathscr{T}_{n,m}$ and $\delta \in (0,\delta_0')$. Then there exists $U_\delta \in \mathrm{U}(n)$ such that
\[
\mathrm{cap}_0(\mathcal{T}_{U_\delta}) \leq \mathrm{cap}(\mathcal{T}) + \|\mathcal{T}\|^m\delta.
\]
By Proposition \ref{p:captoGP} and Theorem \ref{t:GPnearmin} , there is some $y_\delta \in \mathbb{R}^n$ with $|y_\delta| \leq N\log (\frac{1}{\delta})$ and such that
\[
 \sum_{j \in J_{n,m}} d_j(\mathcal{T}_{U_\delta}) e^{\langle y_\delta,u_j \rangle} \leq \mathrm{cap}_0(\mathcal{T}_{U_\delta}) + \delta \max_j (d_j(\mathcal{T}_{U_\delta})).
\]
By \eqref{e:djbound}, note that
\[
d_j(\mathcal{T}_{U_\delta}) \leq C_{n,m}\|\mathcal{T}_{U_\delta}\|^m = C_{n,m}\|\mathcal{T}\|^m. 
\]

Next, set
\[
X_\delta := U_\delta \Lambda_\delta  U_\delta^*,
\]
where $\Lambda_\delta := \textrm{diag}(e^{y_{\delta,1}},\ldots,e^{y_{\delta,n}})$ and $y_{\delta,j}$ denotes the $j$th component of $y_\delta$. Since $|y_\delta| \leq N\log (\frac{1}{\delta})$, we see that the largest eigenvalue of $X_\delta$ is bounded above by $\delta^{-N}$. Similarly, the largest eigenvalue of $X_\delta^{-1}$ is bounded above by $\delta^{-N}$, and therefore $ \max\{\|X_\delta\|,\|X_\delta^{-1}\|\} \leq C_{n,m}\delta^{-N}$. Also
$
\mathcal{T}(X_\delta) = \mathcal{T}_{U_\delta}(\Lambda_\delta)
$
so, by \eqref{e:captoGPmain},
\begin{align*}
   \mathrm{cap}(\mathcal{T};X_\delta) 
   & = \sum_{j \in J_{n,m}} d_j(\mathcal{T}_{U_\delta}) e^{\langle y_\delta,u_j \rangle} \\
   & \leq \mathrm{cap}_0(\mathcal{T}_{U_\delta}) +  C_{n,m}\|\mathcal{T}\|^m \delta \\
   &  \leq \mathrm{cap}(\mathcal{T}) + (1 + C_{n,m})\|\mathcal{T}\|^m\delta.
\end{align*}
By taking $\delta_0 := C_{n,m}\delta_0'$ for an appropriate choice of $C_{n,m}$ we get the claim.
\end{proof}

\begin{remark}
A result similar in nature to Theorem \ref{t:capnearmin} with a diameter bound which depends polynomially on $\delta^{-1}$ was proved by Allen-Zhu--Garg--Li--Oliveira--Wigderson \cite[Theorem 6.1]{AGLOW} under the constraint that the Kraus operators for $\mathcal{T}$ belong to $\mathbb{Z}[i]$ (and in the case $m = n$). This result in \cite{AGLOW} contains more precise information on various constants appearing in the diameter bound; for example, the value of $N$ is polynomial in $n$. On the other hand, in Theorem \ref{t:capnearmin} there is no restriction to integer (or rational) coefficients and the statement is free of explicit reference to the Kraus operators. 

In addition, our proof of Theorem \ref{t:capnearmin} contrasts significantly with the approach in \cite{AGLOW}, the latter being based on a continuous gradient flow argument (see also \cite{KLLR1, KLLR2, KLR} for use of a similar gradient flow). The perspective in the present paper and \cite{AGLOW} also differ in the sense that operator capacity in \cite{AGLOW} is being viewed as a kind of non-commutative version of unconstrained geometric programming. Indeed, as described in \cite[Section 1.2]{AGLOW} and \cite[Section II]{BFGOWW}, there is a wider notion of capacity associated with a continuous group $G$ acting linearly on a finite dimensional vector space $V$. In this framework, unconstrained geometric programming is a commutative example (i.e. arises from the action of a commutative group), whereas operator capacity is a non-commutative example. In the present paper, we shine light on the fact that a commutative perspective can still be effective for operator scaling, at least for results on near-minimisers like Theorem \ref{t:capnearmin} and the forthcoming result on local H\"older regularity (Theorem \ref{t:capHolder}).
\end{remark}

\subsubsection{Local H\"older regularity}

\begin{theorem}\label{t:capHolder}
There exists a constant $\alpha \in (0,1)$, depending only on $n$ and $m$, such that the following holds. For any compact subset $\mathscr{K} \subset {\mathscr{T}_{n,m}}$, there exists a constant \( C < \infty \) such that 
\[
|\mathrm{cap}({\mathcal{T}}) - \mathrm{cap}({\mathcal{T}}')| \leq C \|{\mathcal{T}} - {\mathcal{T}}'\|^{\alpha}
\] 
whenever $\mathcal{T},\mathcal{T}' \in \mathscr{K}$.
\end{theorem}
In particular, Theorem \ref{t:capHolder} implies that for any $K > 0$ there is a constant $C < \infty$ such that
\[
|\mathrm{cap}(\mathcal{T}_\mathcal{A}) - \mathrm{cap}(\mathcal{T}_{\mathcal{A}'})|  \leq C \bigg(\sum_{i=1}^N \|A_i - A_i'\| \bigg)^\alpha 
\]
whenever the Kraus operators $A_i, A_i' \in \mathbb{C}^{m \times n}$ satisfy $\|A_i\|, \|A_i'\| \leq K$ for each $i \in [N]$.

Before proving Theorem \ref{t:capnearmin}, we prepare the following Lipschitz estimate for the coefficients \( (d_j(\mathcal{T}))_{j \in J_{n,m}} \).
\begin{lemma} \label{l:Lipschitz}
There exists a constant $C < \infty$ depending on $n$ and $m$ such that
\[
\| d({\mathcal{T}}) - d({\mathcal{T}}') \|_\infty \leq C R^{m-1} \|{\mathcal{T}} - {\mathcal{T}}'\| 
\]
whenever $R > 0$ and \( \|{\mathcal{T}}\|, \|{\mathcal{T}}'\| \leq R \).
\end{lemma}
It is not difficult to prove Lemma \ref{l:Lipschitz} by combining \eqref{e:djdefn} and \eqref{e:Ddefn}, and so we omit the details. 
\begin{proof}[Proof of Theorem \ref{t:capHolder}]
Let $N \in \mathbb{N}$, $\delta_0 > 0$ and $C_0$ be those that are given by Theorem \ref{t:capnearmin}. Set 
\[
\alpha := \frac{1}{1 + \|\boldsymbol{u}\|N}, 
\]
where  $\|\boldsymbol{u}\| := \max_{j \in J_{n,m}} |u_j|$ and $u_j = j - \frac{m}{n}\boldsymbol{1}_n$.

Suppose $0 < \| \mathcal{T} - \mathcal{T}'\| < \delta_0^{\frac{1}{\alpha}}$ and set $\delta := \| \mathcal{T} - \mathcal{T}'\|^\alpha$. Then $\delta < \delta_0$ and so Theorem \ref{t:capnearmin} provides us with $X \succ 0$ such that $ \max\{\|X\|,\|X^{-1}\|\} \leq \delta^{-N}$ and
\[
\mathrm{cap}(\mathcal{T};X) \leq \mathrm{cap}(\mathcal{T}) + C_0\|\mathcal{T}\|^m\delta.
\]
In particular we have
\begin{align*}
\mathrm{cap}(\mathcal{T}') - \mathrm{cap}(\mathcal{T}) & \leq \mathrm{cap}(\mathcal{T}') - \mathrm{cap}(\mathcal{T};X) + C_0\|\mathcal{T}\|^m\delta \\
& \leq \mathrm{cap}(\mathcal{T}';X) - \mathrm{cap}(\mathcal{T};X) + C_0\|\mathcal{T}\|^m\delta.
\end{align*}
If we write $X = U\Lambda U^*$, where $\Lambda = \mathrm{diag}(e^{y_1},\ldots,e^{y_n})$, by \eqref{e:captoGPmain},
\begin{align*}
\mathrm{cap}(\mathcal{T}';X) - \mathrm{cap}(\mathcal{T};X) = \sum_{j \in J_{n,m}} (d_j(\mathcal{T}_U') - d_j(\mathcal{T}_U))e^{\langle u_j,y\rangle}.
\end{align*}
Hence, by Lemma \ref{l:Lipschitz} and the fact that $ \max\{\|X\|,\|X^{-1}\|\} \leq \delta^{-N}$, we obtain
\begin{align*}
|\mathrm{cap}(\mathcal{T}';X) - \mathrm{cap}(\mathcal{T};X)| & \leq C_{n,m} e^{\|\boldsymbol{u}\| |y| } \|d(\mathcal{T}_U') - d(\mathcal{T}_U)\|_\infty \\
& \leq C_1 \delta^{- \|\boldsymbol{u} \| N} \|\mathcal{T}_U'- \mathcal{T}_U\| \\
& = C_1 \delta^{- \|\boldsymbol{u} \| N} \|\mathcal{T}'- \mathcal{T}\|,
\end{align*}
where $C_1$ is a constant which depends on $n$, $m$ and $\mathscr{K}$. Taking into account our choice of $\alpha$, this yields
\[
\mathrm{cap}(\mathcal{T}') - \mathrm{cap}(\mathcal{T}) \leq (C_1 + C_0\|\mathcal{T}\|^m) \|\mathcal{T}'- \mathcal{T}\|^\alpha.
\]
By the same argument with the roles of $\mathcal{T}$ and $\mathcal{T}'$ switched, we conclude that
\[
|\mathrm{cap}(\mathcal{T}') - \mathrm{cap}(\mathcal{T})| \leq (C_1 + C_0\|\mathcal{T}\|^m) \|\mathcal{T}'- \mathcal{T}\|^\alpha
\]
holds in the case $0 < \| \mathcal{T} - \mathcal{T}'\| < \delta_0^{\frac{1}{\alpha}}$.

In the remaining case where $\| \mathcal{T} - \mathcal{T}'\| \geq \delta_0^{\frac{1}{\alpha}}$, as one would expect the claimed inequality follows more easily. Indeed, \eqref{e:djbound} implies
\[
\mathrm{cap}(\mathcal{T}) \leq \det(\mathcal{T}(I)) \leq C_{n,m}\|\mathcal{T}\|^m
\]
and therefore
\[
|\mathrm{cap}(\mathcal{T}') - \mathrm{cap}(\mathcal{T})| \leq 2C_{n,m}\|\mathcal{T}\|^m \leq 2C_{n,m}\|\mathcal{T}\|^m\delta_0\|\mathcal{T}'- \mathcal{T}\|^\alpha.
\]
Putting everything together we get the claim in Theorem \ref{t:capHolder}.
\end{proof}
The above proof is inspired by the proof of \cite[Theorem 4.5]{BBBCF}. It is in fact possible to deduce Theorem \ref{t:capHolder} from Theorem \ref{t:GPHolder}, but we chose to present the argument above in order to emphasise the role of the diameter bounds on the near-minimisers in Theorem \ref{t:capnearmin}, and since the diameter bounds themselves are interesting in their own right.

\begin{remark}
Using the near-minimiser bounds in \cite[Theorem 6.1]{AGLOW} and the above proof of Theorem \ref{t:capHolder}, one may obtain a result like Theorem \ref{t:capHolder} in the case where the Kraus operators have components in $\mathbb{Z}[i]$. 
\end{remark}

We end this section by observing that our arguments leading to Theorem \ref{t:capHolder} extend to the capacity of quiver data. Following \cite{CD1} and \cite{CD}, for each tuple of matrices
\[
\boldsymbol{V} = \left(V_a \in \mathbb{R}^{n_{i} \times d_{i'}}  \ \middle| \ a \in \mathcal{A}_{i'i}, i' \in [\ell'], i \in [\ell]  \right)
\]
and each $m$-tuple of non-negative exponents $\boldsymbol{\theta}$, define the capacity of the quiver datum $(\boldsymbol{V},\boldsymbol{\theta})$ by
\[
\mathrm{cap}_{\mathcal{Q}}(\boldsymbol{V},\boldsymbol{\theta}) := \inf \left\{ \frac{\prod_{i'=1}^{\ell'} \det(\sum_{i=1}^{\ell} \theta_{i} (\sum_{a \in \mathcal{A}_{i'i}} V_a^* A_{i} V_a))} {\prod_{i=1}^{\ell}  \det(A_{i})^{\theta_{i}}} \ \middle| \ A_{i} \succ 0, i \in [\ell] \right\}.
\]
Here, $\mathcal{Q}$ is the bipartite quiver with source vertices $[\ell']$ and sink vertices $[\ell]$, and the arrows from $i' \in [\ell']$ to $i \in [\ell]$ are labelled by the elements of the set $\mathcal{A}_{i'i}$. Both operator capacity $\mathrm{cap}(\mathcal{T})$ and the inverse-square of the Brascamp--Lieb constant $\mathrm{BL}(\boldsymbol{L},\boldsymbol{\theta})^{-2}$ are special cases of the above capacity for suitable choices of quiver data. 

By using the same line of reasoning in the present paper which led to Proposition \ref{p:captoGP} (that is, diagonalise each $A_{i}$), it is possible to express
$\mathrm{cap}_{\mathcal{Q}}(\boldsymbol{V},\boldsymbol{\theta})$ in terms of unconstrained geometric programming up to an additional minimisation over the compact group $\mathrm{O}(n_{1}) \times \cdots \times \mathrm{O}(n_{\ell})$. Using this, and arguing as in the proof of Theorem \ref{t:capHolder}, we can establish the following regularity result for capacity of quiver data, thus addressing a tentative conjecture in \cite[Remark 14]{CD}. Since the proof is very similar to our earlier arguments, we omit the details.
\begin{theorem}
Let $\boldsymbol{\theta}$ be an $m$-tuple of non-negative exponents. Then the mapping $\boldsymbol{V} \mapsto \mathrm{cap}_{\mathcal{Q}}(\boldsymbol{V},\boldsymbol{\theta})$ is locally H\"older continuous.
\end{theorem}

\begin{remark}
For further context, we also remark that Gurvits--Leake \cite[Corollary 5.8]{GurvitsLeake} have shown that polynomial capacity\footnote{For a given vector $\alpha$ in $\mathbb{R}^n$ with non-negative components, the associated polynomial capacity is given by $\mathrm{cap}_\alpha(p) = \inf_{x > 0} \frac{p(x)}{x^\alpha}$ for $n$-variate polynomials $p$ with non-negative coefficients, and where $x > 0$ means that each component of $x$ is strictly positive.} $p \mapsto \mathrm{cap}_\alpha(p)$ is continuous in the sense of uniform convergence on compact sets.
\end{remark}

\section{The Brascamp--Lieb constant}  \label{section:BLGP}

\subsection{The Brascamp--Lieb constant in terms of geometric programming} 
 
For $i \in [\ell]$, let $L_i : \mathbb{R}^n \to \mathbb{R}^{n_i}$ be a linear mapping and set $\boldsymbol{L} = (L_i)_{i \in [\ell]}$. We refer to $\boldsymbol{L}$ as an $\ell$-transformation. Also, for each $i \in [\ell]$ let $\theta_i$ be a non-negative exponent and $\boldsymbol{\theta} = (\theta_i)_{i \in [\ell]}$. Recall that the pair
$(\boldsymbol{L},\boldsymbol{\theta})$ is said to be a Brascamp--Lieb datum.

That the Brascamp--Lieb constant is expressible in terms of $\Psi$ for rank-one data was first observed by Barthe \cite{Barthe} and later extended to general data in \cite[Theorem 3.1]{BBCF}. We remark that the motivation in \cite{BBCF} to have an expression of this type was to establish continuity of $\boldsymbol{L} \mapsto \mathrm{BL}(\boldsymbol{L},\boldsymbol{\theta})$, and this property of Brascamp--Lieb constant has since found applications in, for example, \cite{BBBCF, BGT}.

In order to state the expression for the Brascamp--Lieb constant in terms of $\Psi$, we introduce some notation. First, set $L := \sum_{i = 1}^\ell n_i$ and $n_0 := 0$. In what follows, we identify $k \in [L]$ with the pair $(i,j)$ by setting
\[
k := n_0 + n_1  + \cdots + n_{i-1} + j
\]
for each $i \in [\ell]$ and $j \in [n_i]$. 

Next, set
\[
\mathcal{I} := \{ I \subseteq [L] : \#(I) = n \}
\]
and write $\mathrm{1}_I \in \mathbb{R}^L$ for the characteristic vector of the subset $I$ of $[L]$. The set $\mathcal{I}$ will be used to index the terms in the geometric program. In particular, the family of vectors \( \boldsymbol{u} = (u_I)_{I \in \mathcal{I}} \subset \mathbb{R}^L \) is given by
\[
u_I := \mathrm{1}_I - T(\boldsymbol{\theta}),
\]
where $T: \mathbb{R}^\ell \hookrightarrow \mathbb{R}^L$ is defined by
\begin{equation} \label{e:Tdefn}
T(\boldsymbol{\theta}) := (\overbrace{\theta_1,\ldots,\theta_1}^{n_1},,\ldots,\overbrace{\theta_\ell,\ldots,\theta_\ell}^{n_\ell}).
\end{equation}
The coefficients $d_I(\boldsymbol{R})$ depend on a collection of orthogonal matrices
\[
\boldsymbol{R} = (R_i)_{i \in [\ell]} \in \mathrm{O}(n_1) \times \cdots \times \mathrm{O}(n_\ell) =: \mathscr{R},
\]
and are given by
\begin{equation} \label{e:dIdefn}
d_I(\boldsymbol{R}) := T(\boldsymbol{\theta})_I\det (L_i^* R_ie_{i,j})_{k \in I}^2.
\end{equation}
Here, $\{e_{i,j} : j \in [n_i] \}$ is the canonical basis of $\mathbb{R}^{n_i}$ and 
\[
T(\boldsymbol{\theta})_I := \prod_{k \in I} T(\boldsymbol{\theta})_k.
\]
\begin{proposition} \label{p:BLtoGP} \cite{BBCF}
For any Brascamp--Lieb datum $(\boldsymbol{L},\boldsymbol{\theta})$, we have   
\begin{equation} \label{e:BLtoGP}
\mathrm{BL}(\boldsymbol{L},\boldsymbol{\theta})^{-2} = \inf_{\boldsymbol{R} \in \mathscr{R}} \Psi(\boldsymbol{d}(\boldsymbol{R})).
\end{equation}
\end{proposition}
Although Proposition \ref{p:BLtoGP} was established in \cite[Theorem 3.1]{BBCF} (see also \cite[Proposition 5.1]{BBBCF}), since the argument is short, we give a brief outline here. For each $A_i \succ 0$, we again use a spectral decomposition $A_i = R_i \Lambda_i R_i^*$ where $R_i \in \mathrm{O}(n_i)$ and $\Lambda_i = \mathrm{diag}(e^{y_{i,1}},\ldots,e^{y_{i,n_i}})$. Then
\begin{equation*}
\sum_{i \in [\ell]} \theta_i L_i^*A_iL_i = \sum_{k \in[L]} T(\boldsymbol{\theta})_k e^{y_k}  v_k(\boldsymbol{R})v_k(\boldsymbol{R})^*
\end{equation*}
where
\[
v_k(\boldsymbol{R}) := L_i^* R_ie_{i,j}.
\]
The Cauchy--Binet formula yields
\begin{equation*}
\det\bigg(\sum_{i \in [\ell]} \theta_i L_i^*A_iL_i\bigg) = \sum_{I \in \mathcal{I}}  d_I(\boldsymbol{R}) e^{\langle y,\mathrm{1}_I \rangle}
\end{equation*}
and this gives \eqref{e:BLtoGP} via Lieb's theorem \eqref{e:Lieb}.

\subsection{The Brascamp--Lieb polyhedron}

For a fixed $\ell$-transformation $\boldsymbol{L}$, we consider the set of admissible exponents for which the Brascamp--Lieb inequality holds; that is, the set
\[
\mathcal{P}(\boldsymbol{{L}}) := \{ \boldsymbol{\theta} : \mathrm{BL}(\boldsymbol{L},\boldsymbol{\theta}) < \infty \}.
\]
To avoid degeneracies, in what follows we assume that each $L_i$ is surjective. 

The following fundamental theorem of Bennett--Carbery--Christ--Tao \cite{BCCT, BCCT_MRL} is a pillar in the theory of the Brascamp--Lieb inequality and widely viewed as the starting point for any investigation into the structure of $\mathcal{P}(\boldsymbol{{L}})$. 
\begin{theorem} \label{t:BCCT} \cite{BCCT, BCCT_MRL}
We have $\boldsymbol{\theta} \in \mathcal{P}(\boldsymbol{{L}})$ if and only if
\begin{equation} \label{e:BLscaling}
\sum_{i=1}^\ell \theta_in_i = n  
\end{equation}
and 
\begin{equation} \label{e:BLtransversal}
\dim(V) \leq \sum_{i=1}^\ell \theta_i\dim(L_iV) \qquad \text{for all $V \leq \mathbb{R}^n$.} 
\end{equation}
\end{theorem}
From this one can quickly infer that $\mathcal{P}(\boldsymbol{{L}})$ is a bounded convex polyhedron and hence a convex polytope. The set $\mathcal{P}(\boldsymbol{{L}})$ is usually referred to as the \emph{Brascamp--Lieb polyhedron}.  

Our main goal in this section is to give a new proof of Theorem \ref{t:BCCT} which rests on the geometric programming representation of the Brascamp--Lieb constant \eqref{e:BLtoGP}. To the best of our knowledge, there are two prior proofs of the sufficiency of \eqref{e:BLscaling} and \eqref{e:BLtransversal} for $\boldsymbol{\theta} \in \mathcal{P}(\boldsymbol{L})$. The proof in \cite{BCCT_MRL} relies on multilinear interpolation and an induction argument using critical subspaces\footnote{Critical subspaces are non-trivial proper subspaces for which equality holds in \eqref{e:BLtransversal}.} -- this argument does not make use of Lieb's theorem. The paper \cite{BCCT} gives a simultaneous proof of Lieb's theorem and the sufficiency of \eqref{e:BLscaling} and \eqref{e:BLtransversal} for $\boldsymbol{\theta} \in \mathcal{P}(\boldsymbol{L})$ using maximisers and heat-flow monotonicity formulae. However, if one assumes Lieb's theorem, then the argument in \cite[Proposition 5.2]{BCCT}, which rests on clever use of linear algebra and compactness, yields the sufficiency of \eqref{e:BLscaling} and \eqref{e:BLtransversal} for $\boldsymbol{\theta} \in \mathcal{P}(\boldsymbol{L})$. We also remark that the ideas in \cite{BCCT} were used by Courtade--Liu \cite{CL} to obtain a finiteness characterisation of the constant in the forward-reverse Brascamp--Lieb inequality, and by Ammari \cite{A} in the setting of regularised Brascamp--Lieb inequalities.

In preparation for our proof of Theorem \ref{t:BCCT}, we begin by giving a reformulation of the condition \eqref{e:BLtransversal} in terms of $L_i^*$ rather than $L_i$. For the statement, we introduce the condition
\begin{equation} \label{e:BCCT'}
\sum_{i=1}^\ell \theta_i\dim(\mathrm{im}(L_i^*) \cap V) \leq \dim(V) \qquad \text{for all $V \leq \mathbb{R}^n$}
\end{equation}
and the condition
\begin{equation} \label{e:BCCT''}
\sum_{i=1}^\ell \theta_i\dim(U_i) \leq \dim\bigg(\sum_{i=1}^\ell U_i \bigg) \qquad \text{for all $U_i \leq \mathrm{im}(L_i^*), i \in [\ell]$.}
\end{equation}
\begin{proposition} \label{p:BCCTalternatives}
The following are equivalent.
\begin{itemize}
    \item[(i)]  \eqref{e:BLscaling} and \eqref{e:BLtransversal}
    \item[(ii)] \eqref{e:BLscaling} and \eqref{e:BCCT'}
    \item[(iii)] \eqref{e:BLscaling} and \eqref{e:BCCT''}
\end{itemize}
\end{proposition}
\begin{proof}
Since $(\ker L_i + V)^\perp = (\ker L_i)^\perp \cap V^\perp$ we get
\[
\dim(\ker L_i \cap V) = \dim(\ker L_i) + \dim(V) - n + \dim(\mathrm{im}(L_i^*) \cap V^\perp).
\]
But the rank-nullity formula applied to the restriction of $L_i$ to $V$ means that the left-hand side of the above identity is equal to $\dim(V) - \dim(L_iV)$, and hence
\[
\dim(L_iV) = n_i - \dim(\mathrm{im}(L_i^*) \cap V^\perp).
\]
It follows that, under the scaling relation \eqref{e:BLscaling}, we have
\[
\dim(V) \leq \sum_{i=1}^\ell \theta_i \dim(L_iV) \quad \Leftrightarrow \quad \sum_{i=1}^\ell \theta_i\dim(\mathrm{im}(L_i^*) \cap V^\perp) \leq \dim(V^\perp)
\]
for any $V \leq \mathbb{R}^n$, and hence (i) is equivalent to (ii).

To see the equivalence between (ii) and (iii), first assume \eqref{e:BCCT'}. For any $U_i \leq \mathrm{im}(L_i^*)$, let $V = \sum_{i=1}^\ell U_i$. Clearly $U_i \leq \mathrm{im}(L_i^*) \cap V$ so we immediately obtain \eqref{e:BCCT''}. Conversely, if we assume \eqref{e:BCCT''} and take any $V \leq \mathbb{R}^n$, by setting $U_i := \mathrm{im}(L_i^*) \cap V$ we have
\[
\sum_{i=1}^\ell \theta_i\dim(\mathrm{im}(L_i^*) \cap V)  \leq \dim\bigg(\sum_{i=1}^\ell U_i \bigg) \leq \dim (V)
\]
where the last inequality holds since $U_i \leq V$ for each $i$.
\end{proof}
\begin{remark}
For rank-two Brascamp--Lieb data, the equivalence of (i) and (ii) in Proposition \ref{p:BCCTalternatives} can be found in Franks--Soma--Goemans \cite[Lemma 7.1]{FSG}. We first became aware of the equivalence of (i), (ii) and (iii) for general data in Ramachandran \cite{R}.
\end{remark}

For a collection of vectors $v_1,\ldots,v_\ell \in \mathbb{R}^n$, the polytope
\[
\mathrm{conv}\{ \mathrm{1}_I : \text{$(v_i)_{i \in I}$ is a basis of $\mathbb{R}^n$}\}
\]
is known as the associated \emph{basis polytope}. In the language of matroid theory, it is known as the vector matroid (or linear matroid) associated with the vectors $v_1,\ldots,v_\ell$. The following theorem of Edmonds \cite{Edmonds} gives a very useful alternative characterisation\footnote{This result appears to have been independently derived elsewhere including, for example, \cite[Theorem 4.4]{CLL}.} of such sets.
\begin{theorem} \label{t:Edmonds}
Let $v_1,\ldots,v_\ell \in \mathbb{R}^n$. Then
\begin{align*}
& \mathrm{conv}\{ \mathrm{1}_I : \text{$(v_i)_{i \in I}$ is a basis of $\mathbb{R}^n$} \} \\
& = \bigg\{x \in \mathbb{R}_+^\ell : \sum_{i=1}^\ell x_i = n \,\, \text{and, for all $I' \subseteq [\ell]$,} \,\, \sum_{i \in I'} x_i \leq \dim(\mathrm{span}(v_i)_{i \in I'})) \bigg\}.
\end{align*}
\end{theorem}
The basis polytopes that we shall we be most interested in are those of the form
\[
K(\boldsymbol{R}) := \mathrm{conv} \{ \mathrm{1}_I : \widetilde{d}_I(\boldsymbol{R}) > 0 \},
\]
where
\[
\widetilde{d}_I(\boldsymbol{R}) := \det (L_i^* R_ie_{i,j})_{k \in I}^2.
\]
Of course, these polytopes depend on $\boldsymbol{L}$ too but we suppress this since the linear mappings under consideration will be clear from the context.
\begin{proposition} \label{p:BLpoly}
For any Brascamp--Lieb datum $(\boldsymbol{{L},\boldsymbol{\theta}})$, the following are equivalent. 
\begin{itemize}
    \item[(i)] $\boldsymbol{\theta} \in \mathcal{P}(\boldsymbol{L})$.
    \item[(ii)] ${\displaystyle T(\boldsymbol{\theta}) \in \bigcap_{\boldsymbol{R} \in \mathscr{R}} K(\boldsymbol{R})}$.
\end{itemize}
\end{proposition}
Essentially, Proposition \ref{p:BLpoly} follows from combining Propositions \ref{p:GPfeasible} and \ref{p:BLtoGP}, but a little care seems to be needed to account for the fact that some components of $\boldsymbol{\theta}$ may vanish. 
\begin{proof}
First suppose $\boldsymbol{\theta} \in \mathcal{P}(\boldsymbol{L})$. Fix $\boldsymbol{R} \in \mathscr{R}$. By Proposition \ref{p:BLtoGP}, $\Psi({\boldsymbol{d}(\boldsymbol{R})}) > 0$ and hence Proposition \ref{p:GPfeasible} implies
\[
T(\boldsymbol{\theta}) \in \mathrm{conv} \{ \mathrm{1}_I :  d_I(\boldsymbol{R}) > 0\}.
\]
This means $T(\boldsymbol{\theta}) \in  K(\boldsymbol{R})$ since $\widetilde{d}_I(\boldsymbol{R}) > 0$ whenever $d_I(\boldsymbol{R}) > 0$.

Conversely, assume that (ii) holds. By Proposition \ref{p:BLtoGP} and continuity of $\Psi$, it suffices to prove $\Psi(\boldsymbol{d}(\boldsymbol{R})) > 0$ for each $\boldsymbol{R} \in \mathscr{R}$. So we fix $\boldsymbol{R} \in \mathscr{R}$ and use (ii) to write
\[
T(\boldsymbol{\theta}) = \sum_{I \in \mathcal{I}} t_I\mathrm{1}_I
\]
where $t_I \in [0,1]$, $\sum_{I \in \mathcal{I}} t_I = 1$ and $t_I = 0$ if $\widetilde{d}_I(\boldsymbol{R}) = 0$. As in the proof of Proposition \ref{p:GPfeasible},
\begin{align*}
\Phi_{\boldsymbol{d}(\boldsymbol{R})}(y) & \geq \prod_{t_I > 0} \bigg(\frac{d_I(\mathbf{R})}{t_I}e^{\langle y,1_I - T(\boldsymbol{\theta})\rangle}\bigg)^{t_I} \\
& = \prod_{t_I > 0} \frac{T(\boldsymbol{\theta})_I \widetilde{d}_I(\boldsymbol{R})}{t_I}
\end{align*}
and hence 
\[
\Psi(\boldsymbol{d}(\boldsymbol{R})) \geq \prod_{t_I > 0} \frac{T(\boldsymbol{\theta})_I \widetilde{d}_I(\boldsymbol{R})}{t_I}.
\]
To conclude, it is enough to check that if $t_I > 0$ then $T(\boldsymbol{\theta})_I > 0$ (we already know that if $t_I > 0$ then $\widetilde{d}_I(\boldsymbol{R}) > 0$). But this follows because if $k \in I \cap \mathrm{supp}(T(\boldsymbol{\theta}))^c$, then
\[
0 = T(\boldsymbol{\theta})_k = \sum_{I' \in \mathcal{I}} t_{I'}\mathrm{1}_{I'}(k) = \sum_{I' \ni k} t_{I'}
\]
and thus $t_I = 0$. This means $I \subseteq \mathrm{supp}(T(\boldsymbol{\theta}))$ whenever $t_I > 0$, and therefore (i) follows.
\end{proof}

\begin{proof}[Proof of Theorem \ref{t:BCCT}]
Suppose first that $\boldsymbol{\theta} \in \mathcal{P}(\boldsymbol{L})$. By Proposition \ref{p:BLpoly} and Theorem \ref{t:Edmonds} we have
\begin{equation} \label{e:E1}
\sum_{k = 1}^L T(\boldsymbol{\theta})_k = n
\end{equation}
and
\begin{equation} \label{e:E2}
\sum_{k \in I'} T(\boldsymbol{\theta})_k \leq \dim(\mathrm{span}(L_i^*R_ie_{i,j}))_{k \in I'}
\end{equation}
for any $I' \subseteq [L]$ and any $\boldsymbol{R} \in \mathscr{R}$. By \eqref{e:E1} we immediately get \eqref{e:BLscaling} and so, by Proposition \ref{p:BCCTalternatives}, it suffices to check \eqref{e:BCCT''}. For this, to each $i \in [\ell]$, we take an arbitrary subspace $U_i$ of $\mathrm{im}(L_i^*)$ of dimension $m_i$ and note that it may be expressed as
\[
U_i = \mathrm{span}(L_i^*\widetilde{e}_{i,j})_{j \in [m_i]},
\]
where $\{\widetilde{e}_{i,j} : j \in [m_i]\}$ are orthonormal vectors in $\mathbb{R}^{n_i}$. Then there exists $R_i \in \mathrm{O}(n_i)$ such that $\widetilde{e}_{i,j} = R_ie_{i,j}$ for all $j \in [m_i]$. We also let $I' \subseteq [L]$ be such that $k \in I'$ if and only if $i \in [\ell]$ and $j \in [m_i]$. By applying \eqref{e:E2} with such $\boldsymbol{R} \in \mathscr{R}$ and $I' \subseteq [L]$, we obtain
\[
\sum_{i = 1}^\ell \theta_i \dim(U_i) = \sum_{k \in I'} T(\boldsymbol{\theta})_k \leq \dim(\mathrm{span}(L_i^*R_ie_{i,j}))_{i \in [\ell], j \in [m_i]}) = \dim\bigg(\sum_{i = 1}^\ell U_i\bigg).
\]

The converse is similar. Assume \eqref{e:BLscaling} and \eqref{e:BLtransversal}. Then
\[
\sum_{k = 1}^L T(\boldsymbol{\theta})_k = \sum_{i = 1}^\ell \theta_in_i = n.
\]
Also, if we take arbitrary $I' \subseteq [L]$ and $\boldsymbol{R} \in \mathscr{R}$, then we define
\[
U_i := \mathrm{span}(L_i^*R_ie_{i,j}))_{j \in J(i)},
\]
where $J(i) := \{ j \in [n_i] : (i,j) \in I'\}$.
Then we use \eqref{e:BCCT''} (which holds thanks to Proposition \ref{p:BCCTalternatives}) to see that
\begin{align*}
 \sum_{k \in I'} T(\boldsymbol{\theta})_k & = \sum_{i = 1}^\ell \theta_i \#(J(i)) \\
 & = \sum_{i = 1}^\ell \theta_i \dim(U_i) \leq  \dim\bigg(\sum_{i = 1}^\ell U_i\bigg) = \dim(\mathrm{span}(L_i^*R_ie_{i,j}))_{k \in I'}.
\end{align*}
It follows from Theorem \ref{t:Edmonds} that 
\[
T(\boldsymbol{\theta}) \in \bigcap_{\boldsymbol{R} \in \mathscr{R}} K(\boldsymbol{R})
\]
and thus, by Proposition \ref{p:BLpoly}, we deduce that $\boldsymbol{\theta} \in \mathcal{P}(\boldsymbol{L})$.
\end{proof}
\begin{remark}
Consider rank-one linear mappings $L_i : \mathbb{R}^n \to \mathbb{R}$ and write $L_ix = \langle x,v_i \rangle$ for some fixed vector $v_i \in \mathbb{R}^n$. In this case, Proposition \ref{p:BLtoGP} reads as
\begin{equation} \label{e:BLrank1GP}
\mathrm{BL}(\boldsymbol{L},\boldsymbol{\theta})^{-2} = \Psi(\boldsymbol{d})
\end{equation}
where
\[
d_I := \boldsymbol{\theta}_I \det(v_i)_{i \in I}^2
\]
and $I$ is a subset of $[\ell]$ of size $n$. If we set $\widetilde{d}_I := \det(v_i)_{i \in I}^2$, then Proposition \ref{p:BLpoly} reduces to 
\[
\mathcal{P}(\boldsymbol{{L}}) = \mathrm{conv}\{ \mathrm{1}_I : \widetilde{d}_I > 0 \}
\]
and this recovers a widely known result of Barthe \cite{Barthe}. We also comment that since $\mathrm{im}(L_i^*) = \mathrm{span}(v_i)$ it is very easy to prove Theorem \ref{t:BCCT} via Edmonds' theorem and Proposition \ref{p:BCCTalternatives}.
\end{remark}

Given the remarkably wide reach of the Brascamp--Lieb inequality, it is highly desirable to have a description of the Brascamp--Lieb polyhedron for a wide class of linear mappings which is as user-friendly as possible. In addition to the areas of mathematics mentioned already earlier in the paper, we also note that there is growing interest in the Brascamp--Lieb inequality from fields such as machine learning and combinatorial optimisation  because certain polytopes can be realised as special cases of Brascamp--Lieb polyhedra (see, for example, \cite{FSG, GGOW_GAFA, HM, HIOS}). In Franks--Soma--Goemans \cite{FSG} it was observed that the Brascamp--Lieb polyhedron in the rank-two case is equal to the perfect fractional matroid matching polytope (introduced by Vande Vate \cite{Vate}). We also refer the interested reader to \cite{FSG} for a description on the complexity of determining membership in Brascamp--Lieb polyhedra.

As far as we are aware, progress on obtaining a tangible expression for the Brascamp--Lieb polyhedron has been restricted to certain very special cases. Taking Theorem \ref{t:BCCT} as a starting point, Valdimarsson \cite{V} established a new characterisation of the Brascamp--Lieb polyhedron in the corank-one case and the mixed rank-one/rank-two case. The mixed rank-one/rank-two case is a remarkable result and even the statement of the result is rather involved. We also note that in recent work of Gressman \cite{Gressman}, there is a novel perspective on the Brascamp--Lieb polyhedron which was inspired by the multilinear factorisation approach  (see Carbery--H\"anninen--Valdimarsson \cite{CHV}). 

With Proposition \ref{p:BLpoly} in mind, one is led to wonder if it may be easier to understand the image of Brascamp--Lieb polyhedra under the embedding $T$ given by \eqref{e:Tdefn}. Clearly
\[
    T(\mathcal{P}(\boldsymbol{L})) \subseteq \bigcap_{\boldsymbol{R} \in \mathscr{R}} K(\boldsymbol{R})
\]
follows from Proposition \ref{p:BLpoly} and it is tempting to think that we have equality
\begin{equation} \label{e:BLpolyQ}
    T(\mathcal{P}(\boldsymbol{L})) = \bigcap_{\boldsymbol{R} \in \mathscr{R}} K(\boldsymbol{R})
\end{equation}
in general. Thanks to Proposition \ref{p:BLpoly}, proving \eqref{e:BLpolyQ} is equivalent to showing 
\begin{equation}
   \mathrm{im}(T) \supseteq \bigcap_{\boldsymbol{R} \in \mathscr{R}} K(\boldsymbol{R}). 
\end{equation}
\begin{remark}
Since $\mathcal{I}$ is a finite set, the right-hand side of \eqref{e:BLpolyQ} is actually the intersection of finitely many basis polytopes $K(\boldsymbol{R})$.
\end{remark}

In the remainder of the paper, we show that \eqref{e:BLpolyQ} holds for the class of $\ell$-transformations considered by Finner \cite{Finner}. Certain other classes of transformations may be handled in a similar manner, but we leave open the problem of determining exactly which data satisfy \eqref{e:BLpolyQ}.
\begin{theorem}
Suppose $S_i \subseteq [n]$ with $\#(S_i) = n-n_i$ and $L_i : \mathbb{R}^n \to \mathbb{R}^{n_i}$ omits the variables with components in $S_i$. Then \eqref{e:BLpolyQ} holds.    
\end{theorem}
\begin{proof}
First observe that 
\[
\{ L_i^*e_{i,j} : i \in [n],j \in [n_i] \} \subseteq \{e_1,\ldots,e_n\},
\]
where $\{e_1,\ldots,e_n\}$ is the canonical basis of $\mathbb{R}^n$. From the list of vectors $\{ L_i^*e_{i,j} : i \in [n],j \in [n_i] \}$ we get a partition 
\[
[L] = \bigcup_{\tau = 1}^n \Omega_\tau
\]
by setting
\[
\Omega_\tau := \{ k \in [L] : L_i^*e_{i,j} = e_\tau\}.
\]
Furthermore, note that
\[
\{L_i^*e_{i,j} : j \in [n_i]\}
\]
contains $n_i$ distinct vectors.

Let us show first that $x_1 = \cdots = x_{n_1}$ and begin by considering the case $\boldsymbol{R} = (\mathrm{id},\ldots,\mathrm{id})$. For $I \in \mathcal{I}$ to be such that $d_I(\boldsymbol{R}) > 0$, we need to take precisely one element of each $\Omega_\ell$ and put it in $I$. Hence
\[
K(\boldsymbol{R}) = \mathrm{conv}\{  \mathrm{1}_I : \#(I \cap \Omega_\tau) = 1 \,\, \text{for all $\tau \in [n]$}\}.
\]
From this we see in particular that
\begin{equation} \label{e:K}
K(\boldsymbol{R}) \subseteq \bigg\{ x \in [0,1]^L : \sum_{k \in \Omega_\tau} x_k = 1 \,\, \text{for all $\tau \in [n]$} \bigg\}.
\end{equation}
Indeed, if $I \in \mathcal{I}$ is such that $\#(I \cap \Omega_\tau) = 1$ for all $\tau \in [n]$, then
\[
\sum_{k \in \Omega_\tau} \mathrm{1}_I(k) = \sum_{k \in I \cap \Omega_\tau} 1 = 1
\]
and therefore \eqref{e:K} follows because the set on the right-hand side is convex.

Next consider the case $\boldsymbol{R}' = (\Sigma,\mathrm{id},\ldots,\mathrm{id})$, where $\Sigma \in \mathrm{O}(n_1)$ is chosen as follows.
Write
\[
\{L_1^*e_{1,j} : j \in [n_1]\} = \{e_\tau : \tau \in E\}
\]
where $E \subseteq [n]$ and $\#(E) = n_1$, and let $\iota : E \to [n_1]$ be the induced bijection; i.e. the bijection given by
\[
e_\tau = L_1^*e_{1,\iota(\tau)} \qquad (\tau \in E).
\]
Observe that this means $\iota(\tau) \in \Omega_\tau$ whenever $\tau \in E$.

Let $\sigma$ be any permutation of $[n_1]$ which has no fixed points, and associated with this permutation we let $\Sigma \in \mathrm{O}(n_1)$ be such that
\[
\Sigma(e_{1,\sigma(j)}) = e_{1,j} \qquad (j\in [n_1]). 
\]
Then of course
\[
\{L_1^*\Sigma e_{1,j} : j \in [n_1]\} = \{e_\tau : \tau \in E\}.
\]
Also, from the list of vectors 
\begin{equation} \label{e:rotatedlist}
\{ L_1^*\Sigma e_{1,j} : j \in [n_1] \} \cup \{ L_i^*e_{i,j} : i \in [n] \setminus \{1\},j \in [n_i] \}    
\end{equation}
we get a partition
\[
[L] = \bigcup_{\tau = 1}^n \Omega_\tau'
\]
by setting
\[
\Omega_\tau' := 
\left\{
\begin{array}{lll}
\Omega_\tau & \text{if $\tau \notin E$} \\
(\Omega_\tau \setminus \{\iota(\tau)\}) \cup \{\sigma(\iota(\tau))\} & \text{if $\tau \in E$.}
\end{array}
\right.
\]
Note that, by construction, if $\tau \in E$ then
\[
e_\tau = L_1^*\Sigma e_{1,\sigma(\iota(\tau))}
\]
and so $\sigma(\iota(\tau))$ plays the role of $\iota(\tau)$ above (i.e. in the first block of the list \eqref{e:rotatedlist}, $e_\tau$ appears in the $\sigma(\iota(\tau))$th position). Also, as above, we have
\begin{align*}
K(\boldsymbol{R}') & = \mathrm{conv}\{  \mathrm{1}_I : \#(I \cap \Omega_\tau') = 1  \,\, \text{for all $\tau \in [n]$}\} \\
& \subseteq \bigg\{ x \in [0,1]^L : \sum_{k \in \Omega_\tau'} x_k = 1 \,\, \text{for all $\tau \in [n]$} \bigg\}.
\end{align*}

Now take an arbitrary $x \in K(\boldsymbol{R}) \cap K(\boldsymbol{R}')$, and let $\tau \in E$. Then
\[
\sum_{k \in \Omega_\tau} x_k = 1 = \sum_{k \in \Omega_\tau'} x_k = 1
\]
and therefore
\[
x_{\iota(\tau)} = x_{\sigma(\iota(\tau))}.
\]
Since $\iota$ is a bijection, this means $x_j = x_{\sigma(j)}$ holds for all $j \in [n_1]$, and hence $x_1 =  \cdots = x_{n_1}$. By similar arguments, we can obtain $x_{n_{i-1} + 1} =  \cdots = x_{n_i}$ for each $i \in [\ell]$, and thus $x \in \mathrm{im}(T)$.
\end{proof}

\begin{remark}
In fact, \eqref{e:K} holds with equality.
To check this, first note that for any $I \subseteq [L]$, it holds that
    \begin{equation}\label{e:Identity}
    \# \{ \tau \in [n]\, :\, \Omega_\tau \cap I \neq \emptyset \} = \dim ({\rm span}(v_k)_{k \in I}), 
    \end{equation}
    where $v_k := L_i^* e_{i,j}$ for $k=(i,j) \in [n] \times [n_i]$. Indeed, for $k, k' \in I$, the vectors $v_k$ and $v_{k'}$ coincide if and only if $k,k' \in \Omega_\tau$ for some $\tau \in [n]$.
    From this we obtain \eqref{e:Identity}. 

    Now take any $x \in [0,1]^L$ with 
    \begin{equation}\label{e:Identity2}
    \sum_{k \in \Omega_\tau} x_k =1 
    \end{equation}
    for all $\tau \in [n]$. To show the opposite inclusion to \eqref{e:K} we use Theorem \ref{t:Edmonds}. For this, take arbitrary $I \subseteq [L]$ and use \eqref{e:Identity} and \eqref{e:Identity2} to get
     \begin{equation*}
        \sum_{k \in I} x_k
        =
        \sum_{\substack{\tau \in [n] \\ \Omega_\tau \cap I\neq \emptyset} } \sum_{k \in \Omega_\tau \cap I} x_k
        \le
        \sum_{\substack{\tau \in [n] \\ \Omega_\ell \cap I \neq \emptyset}} \sum_{k \in \Omega_\ell} x_k
        =
        \dim ({\rm span}(v_k)_{k \in I}).
    \end{equation*}
   In a similar way 
    $$
    \sum_{k \in [L]} x_k
        =
        \sum_{\tau \in [n]  } \sum_{k \in \Omega_\tau} x_k
        =
        n. 
    $$
    The assertion
    \begin{equation*} 
K(\boldsymbol{R}) \supseteq \bigg\{ x \in [0,1]^L : \sum_{k \in \Omega_\tau} x_k = 1 \,\, \text{for all $\tau \in [n]$} \bigg\}.
\end{equation*}
    now follows from Theorem \ref{t:Edmonds}.
    \end{remark}




\begin{thebibliography}{MMMMM}
\bibitem{AGLOW} Z. Allen-Zhu, A. Garg, Y. Li, R. Oliveira, A. Wigderson, \textit{Operator scaling via geodesically convex optimization, invariant theory and polynomial identity testing}, arXiv:1804.01076.

\bibitem{A} B. Ammari, \textit{Regularized Brascamp--Lieb inequalities via optimal transport and study of equality cases}, arXiv::2603.21267.

\bibitem{AFR} K. Astala, D. Faraco, K. Rogers, \textit{On Plancherel’s identity for a two-dimensional
scattering transform}, Nonlinearity \textbf{28} (2015), 2721--2729.

\bibitem{Ball} K. Ball, \textit{Volumes of sections of cubes and related problems}, Geometric Aspects of Functional Analysis, Springer Lecture Notes in Math. 1376 (1989) 251--260.

\bibitem{Bapat} R. B. Bapat, \textit{Mixed discriminants of positive semidefinite matrices}, Linear Algebra Appl. \textbf{126} (1989), 107--124.

\bibitem{Barthe} F. Barthe, \textit{On a reverse form of the Brascamp--Lieb inequality}, Invent. Math. \textbf{134} (1998), 335--361.

\bibitem{BGMN} F. Barthe, O. Gu\'edon, S. Mendelson, A. Naor, \textit{A probabilistic approach to the
geometry of the $\ell^p$-ball}, Ann. Probab. \textbf{33} (2005) 480--513.

\bibitem{BBBCF} J. Bennett, N. Bez, S. Buschenhenke, M. G. Cowling, T. C. Flock, \textit{On the nonlinear Brascamp--Lieb inequality}, Duke Math. J. \textbf{169} (2020), 3291--3338.

\bibitem{BBCF} J. Bennett, N. Bez, M. G. Cowling, T. C. Flock, \textit{Behaviour of the Brascamp--Lieb constant}, Bull. Lond. Math. Soc. \textbf{49} (2017), 512--518.

\bibitem{BCCT} J. Bennett, A. Carbery, M. Christ, T. Tao,  \textit{The Brascamp--Lieb inequalities: finiteness, structure and extremals}, Geom. Funct. Anal. \textbf{17} (2008), 1343--1415.

\bibitem{BCCT_MRL} J. Bennett, A. Carbery, M. Christ, T. Tao,  \textit{Finite bounds for H\"older--Brascamp--Lieb multilinear inequalities}, Math. Res. Lett. \textbf{17} (2010), 647--666.

\bibitem{BGT} N. Bez, A. Gauvan, H. Tsuji, \textit{A note on ubiquity of geometric Brascamp--Lieb data}, Bull. Lond. Math. Soc. \textbf{57} (2025), 302--314.

\bibitem{Brazitikos} S. Brazitikos, \textit{Brascamp--Lieb inequality and quantitative versions of Helly’s theorem},
Mathematika \textbf{63} (2017), 272--291.

\bibitem{BV} S. Boyd, L. Vandenberghe, \textit{Convex optimization}, Cambridge University Press, 2004.

\bibitem{B} R. Brown, \textit{Estimates for the scattering map associated with a two-dimensional first-order
system}, J. Nonlinear Sci. \textbf{11} (2001) 459--471.

\bibitem{BLNW} 
P. B\"urgisser, Y. Li, H. Nieuwboer, M. Walter,
\textit{Interior-point methods for unconstrained geometric programming and scaling problems}, arXiv:2008.12110.

\bibitem{BFGOWW} 
P. B\"urgisser, C. Franks, A. Garg, R. Oliveira, M. Walter, A. Wigderson, \textit{Towards a theory of non-commutative optimization: geodesic 1st and 2nd order methods for moment maps and polytopes}, 2019 IEEE 60th Annual Symposium on Foundations of Computer Science, 845--861.

\bibitem{CHV} A. Carbery, T. H\"anninen, S. Valdimarsson, \textit{Multilinear duality and factorisation for Brascamp--Lieb-type inequalities}, J. Eur. Math. Soc. \textbf{25} (2023), 2057--2125.

\bibitem{CC-E} E. Carlen, D. Cordero-Erausquin, \textit{Subadditivity of the entropy and its relation to
Brascamp--Lieb type inequalities}, Geom. Funct. Anal. \textbf{19} (2009), 373--405.

\bibitem{CLL} E. A. Carlen, E. H. Lieb, M. Loss \textit{A sharp analog of Young's inequality on $S^N$ and related entropy inequalities}, J. Geom. Anal. \textbf{14} (2004), 487--520.

\bibitem{C} M. Chiang, \textit{Geometric programming for communication systems}, Communications and Information Theory \textbf{2} (2005), 1--154.

\bibitem{CD1} C. Chindris, H. Derksen, \textit{The capacity of quiver representations and Brascamp--Lieb constants}, Int. Math. Res. Not. IMRN \textbf{2022} (2022), 19399--19430.

\bibitem{CD} C. Chindris, H. Derksen, \textit{Algebraicity of the Brascamp--Lieb constants}, arXiv:2603.09057.


\bibitem{Choi} M. D. Choi, \textit{Completely positive linear maps on complex matrices}, Linear Algebra Appl. \textbf{10} (1975), 285--290.



\bibitem{CL} T. Courtade, J. Liu, \textit{Euclidean forward–reverse Brascamp--Lieb Inequalities: finiteness, structure, and extremals}, J. Geom. Anal. \textbf{31} (2021), 3300--3350.


\bibitem{DPZ} R. J. Duffin, E. L. Peterson, C. Zener, \textit{Geometric programming - theory
and application}, John Wiley \& Sons, 1967.

\bibitem{Edmonds} J. Edmonds, Submodular functions, matroids, and certain polyhedra, in Combinatorial
Structures and Their Applications, Gordon and Breach, New York, 1970, pp. 69--87.

\bibitem{Finner} H. Finner, A generalization of H\"older's inequality and some probability inequalities, Ann. Probab. \textbf{20}
(1992), 1893--1901.

\bibitem{Franks} C. Franks, \textit{Operator scaling with specified marginals}, STOC'18---Proceedings of the 50th Annual ACM SIGACT Symposium on Theory of Computing, 190--203.

\bibitem{FSG} C. Franks, T. Soma, M. X. Goemans, \textit{Shrunk subspaces via operator sinkhorn iteration}, Proceedings of the 2023 Annual ACM-SIAM Symposium on Discrete Algorithms (SODA), 1655–1668. Society for Industrial and Applied Mathematics (SIAM), Philadelphia, PA, 2023.

\bibitem{GGOW_GAFA} A. Garg, L. Gurvits, R. Oliveira, A. Wigderson, \textit{Algorithmic and optimization aspects of Brascamp--Lieb inequalities, via Operator Scaling}, Geom. Funct. Anal. \textbf{28} (2018) 100--145.

\bibitem{GGOW_FOCM} A. Garg, L. Gurvits, R. Oliveira, A. Wigderson, \textit{Operator scaling: theory and applications}, Found. Comput. Math.  \textbf{20} (2020) 223--290.

\bibitem{Gressman} P. T. Gressman, \textit{Finiteness of the H\"older--Brascamp--Lieb constant revisited}, arXiv:2503.15656.

\bibitem{GZ} S. Guo, R. Zhang, \textit{On integer solutions of Parsell--Vinogradov systems}, Invent. Math. \textbf{218} (2019), 1--81.

\bibitem{Gurvits} L. Gurvits, \textit{Classical complexity and quantum entanglement}, J. Comput. System Sci. \textbf{69} (2004), 448--484.

\bibitem{GurvitsLeake} L. Gurvits, J. Leake, \textit{Counting matchings via capacity-preserving operators}, Combin. Probab. Comput. \textbf{30} (2021), 956--981.

\bibitem{GS} L. Gurvits, A. Samorodnitsky, \textit{A deterministic algorithm for approximating the mixed discriminant and mixed volume,
and a combinatorial corollary}, Discrete Comput. Geom. \textbf{27} (2002), 531--550.

\bibitem{HM} L. Hamilton, A. Moitra, \textit{The Paulsen problem made simple}, Israel J. Math. \textbf{246} (2021), 299--313.

\bibitem{HM} M. Hardt, A. Moitra, \textit{Algorithms and Hardness for Robust Subspace Recovery}, Proceedings of the 26th Annual Conference on Learning Theory, Proceedings of Machine Learning Research \textbf{30} (2013), 354--375.


\bibitem{HIOS} H. Hirai, Y. Iwamasa, T. Oki, T. Soma, \textit{Algebraic combinatorial optimization on the degree of determinants of noncommutative symbolic matrices}, Math. Program. \textbf{213} (2025), 941--984.

\bibitem{KLLR1} T. C. Kwok, L. C. Lau, Y. T. Lee and A. Ramachandran, \textit{The Paulsen problem, continuous operator scaling, and smoothed analysis}, in STOC’18-Proceedings of the 50th Annual ACM SIGACT Symposium on Theory of Computing, ACM, New York, 2018, pp. 182--189.

\bibitem{KLLR2} T. C. Kwok, L. C. Lau, Y. T. Lee and A. Ramachandran, \textit{The Paulsen problem, continuous operator scaling, and smoothed analysis}, arXiv:1710.02587.


\bibitem{KLR} T. C. Kwok, L. C. Lau, A. Ramachandran, \textit{Spectral analysis of matrix scaling and operator scaling}, SIAM J. Comput. \textbf{50} (2021), 1034--1102.



\bibitem{Lieb} E. H. Lieb, \textit{Gaussian kernels have only Gaussian maximizers}, Invent. Math. \textbf{102} (1990), 179--208.

\bibitem{P} P. Perry, \textit{Global well-posedness and long-time asymptotics for the defocussing Davey--Stewartson II equation in $H^{1,1}(\mathbb{C})$, with an appendix by M. Christ}, J. Spectral Theory \textbf{6} (2016), 429--481.

\bibitem{Petrow} I. Petrow, \textit{The Weyl law for algebraic tori}, J. Eur. Math. Soc. \textbf{26} (2024), 2441--2532.

\bibitem{R} A. Ramachandran, \textit{The Brascamp--Lieb Inequality: Matroid Matching and Rank of Matrix Spaces}, slides available at https://algcomp.uwaterloo.ca/seminar-materials/2022-10-21-akshay.pdf

\bibitem{SV} D. Straszak, N. K. Vishnoi, \emph{Maximum entropy distributions: Bit
complexity and stability}, Proceedings of the Thirty-Second Conference on Learning
Theory, volume 99 of Proceedings of Machine Learning Research, pages 2861--2891.

\bibitem{V} S. Valdimarsson, \textit{The Brascamp--Lieb polyhedron}, Canad. J. Math. \textbf{62} (2010), 870--888.

\bibitem{Vate} J. H. Vande Vate, \textit{Fractional matroid matchings} J. Combinat. Theory Ser. B \textbf{55} (1992), 133--145.

\bibitem{Zhang} R. Zhang, \emph{The Brascamp--Lieb inequality and its influence on Fourier analysis}, in The
Physics and Mathematics of Elliott Lieb, Vol. 2, 585--628.
\end{thebibliography}
\end{document}